\documentclass{amsart}
\usepackage{amsmath,amsthm,amssymb,amsfonts}
\usepackage{stmaryrd}
\usepackage{array}
\usepackage{tabularx}
\usepackage{graphicx}
\usepackage{hyperref}
\usepackage[all,cmtip]{xy}

\newcommand{\Ext}{\operatorname{Ext}}
\newcommand{\Tor}{\operatorname{Tor}}
\newcommand{\Hom}{\operatorname{Hom}}

\newcommand{\HH}{{\operatorname H}}
\newcommand{\dell}{\partial}

\newcommand{\fm}{{\mathfrak{m}}}

\newcommand{\Ann}{\operatorname{Ann}}

\newcommand{\grade}{\operatorname{grade}}

\newcommand{\rank}{\operatorname{rank}}

\newcommand{\depth}{\operatorname{depth}}

\newcommand{\pd}{\operatorname{pd}}

\newcommand{\fd}{\operatorname{fd}}
\newcommand{\ch}{\operatorname{char}}

\newcommand{\ZZ}{\mathbb{Z}}
\newcommand{\NN}{\mathbb{N}}

\newcommand{\ra}{\rightarrow}

\newcommand{\xra}{\xrightarrow}

\newcommand{\var}{{\hskip1pt\vert\hskip1pt}}
\newcommand{\dbl}{[\hspace{-1.5pt}[}
\newcommand{\dbr}{]\hspace{-1.5pt}]}

\swapnumbers
\theoremstyle{plain}
\newtheorem{theorem}{Theorem}[section]

\newtheorem{lemma}[theorem]{Lemma}
\newtheorem{corollary}[theorem]{Corollary}
\newtheorem{fact}[theorem]{Fact}

\theoremstyle{definition}
\newtheorem{definition}[theorem]{Definition}

\newtheorem{example}[theorem]{Example}

\newtheorem{construction}[theorem]{Construction}

\theoremstyle{remark}
\newtheorem{remark}[theorem]{Remark}

\newcolumntype{C}{>{\centering\arraybackslash}X}

\usepackage{hyperref}
\hypersetup{
    bookmarks=true,
    unicode=false,
    pdftoolbar=true,
    pdfmenubar=true,
    pdffitwindow=true,
    pdfstartview={FitH},
    pdftitle={Existence of Totally Reflexive Modules via Gorenstein Homomorphisms},
    pdfauthor={K. A. Beck},
    pdfsubject={},
    pdfcreator={},
    pdfproducer={},
    pdfkeywords={},
    pdfnewwindow=true,
    colorlinks=false,
    linkcolor=red,
    citecolor=green,
    filecolor=magenta,
    urlcolor=cyan
}

\begin{document} \title[Existence of Totally Reflexive Modules]{Existence of Totally Reflexive Modules via Gorenstein Homomorphisms}
\author[K.\ A.\ Beck]{Kristen A.\ Beck} 
\address{Kristen A.\ Beck, Department of Mathematics, University of Texas at Arlington, Box 19408, Arlington, TX 76019, U.\ S.\ A.} 
\email{kbeck@uta.edu}

\thanks{{\em Date.} \today.} 

\maketitle

\begin{abstract} We define, via Gorenstein homomorphisms, a class of local rings over which there exist
non-trivial totally reflexive modules.  We also provide a general construction of such rings, which
indicates their abundance. \end{abstract}

\section*{Introduction}
In 1967, Auslander first introduced the notion of a totally reflexive module when he defined Gorenstein dimension in \cite{Au}.  Since that time, totally reflexive modules have been studied extensively.  The goal of this paper is to investigate the existence of such modules.  On one hand, their existence is always guaranteed, as every projective module is trivially totally reflexive.  A more interesting problem, however, is to understand the prevalence of non-trivial (that is, non-projective) totally reflexive modules.  It is well known that over a Gorenstein local ring, it is precisely the maximal Cohen-Macaulay modules which are totally reflexive; when the ring is also non-regular, these modules are non-free.  However, over a non-Gorenstein ring, non-trivial totally reflexive modules are much more elusive.

There has been some work done to answer this existence question over non-Gorenstein rings.  Among others to study this problem (see, for example, \cite{AM}, \cite{T}, \cite{TW}, \cite{Yo}),  Avramov, Gasharov, and Peeva showed in \cite[Theorem 3.2]{AGP} that any local ring $R$ with an embedded deformation $S\to S/(x_1,\ldots,x_n)\cong R$, with $x_1,\ldots,x_n$ an $S$-regular sequence, admits non-trivial totally reflexive modules.  In this paper, we weaken the hypotheses of \cite[Theorem 3.2]{AGP} by considering a local ring homomorphism $S\to S/I\cong R$ where $I$ is a Gorenstein ideal of $S$, rather than an ideal generated by a regular sequence.  In this generality one would not expect such an $R$ to always admit non-trivial totally reflexive modules.  However, we prove in our main theorem that when a particular lifting of $R$ to a Gorenstein ring exists, the desired result is obtained.  As regular sequences always lift modulo certain surjective ring homomorphisms, our result recovers that of \cite{AGP}, mentioned above.  This result, along with several corollaries, is given in Section 2.

In Section 3 we provide a general construction for rings of arbitrary grade which satisfy the hypotheses of our result, and therefore admit non-trivial totally reflexive modules.  Subsequently, we give a specific example which, using machinery established in Section 4, we show does not have an embedded deformation.

\section{Preliminaries}
Throughout, let $(R,\fm,k)$ denote a commutative local (meaning also Noetherian) ring with maximal ideal $\fm$ and residue class field $k=R/\fm$.  All modules will be finitely generated, and we shall reserve the symbol $\mu_R(M)$ for the minimal number of generators of the $R$-module $M$.  Furthermore, by $(-)^*$ we denote the ring dual $\Hom_R(-,R)$.
\subsection{Total reflexivity}
Let $M$ be a finitely generated module over $R$.  Then $M$ is said to be \emph{totally reflexive} if each of the following conditions hold:
\begin{enumerate}
\item The canonical map $M\to\Hom_R(M^*,R)$ is an isomorphism.
\item $\Ext_R^i(M,R)=0$ for all $i>0$.
\item $\Ext_R^i(M^*,R)=0$ for all $i>0$.
\end{enumerate}
Moreover, we say that $M$ is \emph{non-trivial} if it is not free.

Recall that an acyclic complex $\mathbf{C}=(C_i,\dell_i)_{i\in\ZZ}$ of finitely generated free $R$-modules is called \emph{minimal} if $\dell_i(C_i)\subseteq\fm C_{i-1}$ for all $i\in\ZZ$ (for example, see \cite[Proposition 8.1]{AM}).  Furthermore, if $\mathbf{C}$ is such that $\mathbf{C}^*=(C_i^*,\dell_i^*)_{i\in\ZZ}$ is also acyclic, then it is said to be a (\emph{minimal}) \emph{totally acyclic complex}.  Any totally acyclic complex which is minimal and nonzero is \emph{non-trivial}.

Given a totally acyclic complex $\mathbf{C}=(C_i,\dell_i)_{i\in\ZZ}$ of $R$-modules, we have that $\Omega^i\mathbf{C}=\ker\dell_i$ is a totally reflexive $R$-module for every $i\in\ZZ$.  Conversely, it is easy to see that every totally reflexive module gives rise to a totally acyclic complex.  Thus, the existence of totally reflexive modules and totally acyclic complexes are equivalent notions.

\subsection{Gorenstein homomorphisms}
For finitely generated $Q$-modules $M$ and $N$, we define the \emph{grade of $N$ on $M$} by
\[
\grade_Q(N,M)=\min\{i\var\Ext_Q^i(N,M)\neq 0\}.
\]
Note that this quantity is well-defined over a non-local ring.  In fact, the quantity $\grade_Q(N,M)$ is simply the common length of the maximal $M$-regular sequences contained in $\Ann_Q N$.  We also speak of the \emph{grade} of a finitely generated $Q$-module $M$, which is defined by
\[
\grade_Q M=\min\{i\var\Ext_Q^i(M,Q)\neq 0\}.
\]
Thus, $\grade_Q M$ is simply $\grade_Q(M,Q)$, as defined above.

A finitely generated $Q$-module $M$ is called \emph{perfect} if $\pd_Q M=\grade_Q M$.  It is
easy to see that if $Q$ is Cohen-Macaulay and $\pd_Q M<\infty$, then $M$ is perfect if and only if it is
Cohen-Macaulay.  Moreover, let $I\subseteq Q$ be an ideal such that $Q/I$ is a perfect $Q$-module. Then $I$ is
called a \emph{Gorenstein ideal} if
\[
\Ext_Q^{\pd_Q I}(Q/I,Q)\cong Q/I.
\]
Notice that if $Q$ is Gorenstein and $R/I$ is a perfect $Q$-module, then $I$ is a Gorenstein ideal if and only if $Q/I$ is a Gorenstein ring.

A surjective ring homomorphism $\varphi:Q\ra R$ is called \emph{Cohen-Macaulay} if $R$ is a
perfect $Q$-module.  Moreover, $\varphi$ is called \emph{Gorenstein} if it is perfect and $\ker\varphi$ is a Gorenstein ideal of $Q$.   With these definitions, notice that the image of a Cohen-Macaulay ring under a Cohen-Macaulay ring homomorphism is
also Cohen-Macaulay.  Moreover, the image of a Gorenstein ring under a Gorenstein ring homomorphism is
Gorenstein.

\subsection{Lifting}
Let $\varphi:P\ra Q$ be ring homomorphism and $M$ a finitely generated $Q$-module. If
there exists a finitely generated $P$-module $N$ such that \begin{enumerate} \item $M\cong Q\otimes_P N$ and \item
$\Tor_i^P(Q,N)=0$ for all $i>0$ \end{enumerate} then $M$ is said to \emph{lift} to $P$ via $Q$.  In this
case, $N$ is called a \emph{lifting} of $M$.

\section{Existence of Non-trivial Totally Acyclic Complexes}

The framework of our main theorem is motivated by the phenomenon apparent in the most general class of rings
previously known to always admit non-trivial totally acyclic complexes: local rings having an embedded
deformation. The next result is a consequence of the explicit construction in \cite[Theorem 3.2]{AGP}
due to Avramov, Gasharov, and Peeva.  First we give the definition.

\begin{definition} A local ring $R$ is said to have an \emph{embedded deformation} if there exists a local
ring $S$ and an $S$-regular sequence $x_1,\ldots,x_n\subseteq\fm_S^2$ such that $R\cong S/(x_1,\ldots,x_n)$.
\end{definition}

\begin{theorem}\label{thm:embdef}\cite[Theorem 3.2]{AGP} Let $R$ be a local ring which has an embedded deformation.  Then there exist
non-trivial totally acyclic complexes over $R$. \end{theorem}

This result will prove to be corollary to our main theorem.  The proof we give below is completely different than the construction in \cite{AGP}.

\medskip

The method used in our main result relies on the following result concerning the descent of totally reflexive modules along homomorphisms of finite flat dimension, cf.\ \cite[5.6(b)]{CH}.

\begin{lemma}\label{lem:finflatdim} Let $Q$ be a non-regular Gorenstein ring and $\varphi:Q\ra R$ a local homomorphism of finite flat dimension. Then there exist non-trivial totally acyclic complexes over $R$. \end{lemma}

\begin{proof} Since $Q$ is Gorenstein and non-regular, it admits non-trivial minimal totally acyclic complexes;
let $\mathbf{C}=(C_i,\dell_i)_{i\in\ZZ}$ be such a complex, and consider the totally reflexive module $\Omega^j\mathbf{C}$ for some $j\in\ZZ$. If $\fd_QR=n$, we have that $\Tor^Q_i(\Omega^j\mathbf{C},R)$ vanishes for all $i>n$.  Letting $j$ vary, it follows that $\mathbf{C}\otimes_Q R$ is exact over $R$.  Furthermore, the isomorphism given by
\[
\Hom_R(\mathbf{C}\otimes_Q R,R)\cong\Hom_Q(\mathbf{C},Q)\otimes_Q R
\]
(cf.\ \cite[Proposition 2.3]{JM}) implies that the same argument shows that $\left(\mathbf{C}\otimes_Q R\right)^*=\Hom_R(\mathbf{C}\otimes_Q R,R)$ is also exact over $R$. Therefore $\mathbf{C}\otimes_Q R$ is a totally acyclic complex
over $R$, and it is non-trivial and minimal since $\mathbf{C}$ was. \end{proof}

We can now state and prove our main theorem.

\begin{theorem}\label{thm:main} Let $\varphi:Q\ra R$ be a Gorenstein homomorphism of local rings whose kernel is
contained in $\fm_Q^2$.  Suppose that there exists a Gorenstein local ring $P$ and homomorphism $\psi:P\ra
Q$ of finite flat dimension. Furthermore let $S=P/I$ be a lifting of $R$ to $P$ via $Q$ such that
\begin{enumerate} \item $0\neq I\subseteq\fm_P^2$ and \item $\grade_P(S,P)\geq\grade_P(S,Q)$. \end{enumerate} Then
there exist non-trivial totally acyclic complexes over $R$. \end{theorem}

\begin{proof} Consider the natural projection $\varphi':P\ra S$, and the map $\psi':S\ra
S\otimes_P Q\cong R$ which acts by $s\mapsto s\otimes 1$.  Then we have the following commutative diagram of
local ring homomorphisms:
\begin{equation*}\tag{\dag}
\begin{split}
\xymatrixrowsep{2pc} \xymatrixcolsep{2pc}
\xymatrix{
& P\ar@{->}_{\varphi'}[dl]\ar@{->}^{\psi}[dr] & \\
S\ar@{->}_{\psi'}[dr] & &
Q\ar@{->}^{\varphi}[dl] \\
& R & }
\end{split}
\end{equation*}

By Lemma \ref{lem:finflatdim}, it suffices to show that $S$ is a non-regular Gorenstein ring and that $\fd\psi'<\infty$.

Let $\mathbf{F}$ be a minimal free resolution of $S$ over $P$.  Since $\Tor_i^P(S,Q)$ vanishes for each $i>0$, a
minimal free resolution of $R$ over $Q$ is given by $\mathbf{F}\otimes_PQ$, and therefore $\pd_P S=\pd_Q R<\infty$.
Also by the vanishing of $\Tor_i^P(S,Q)$ for each $i>0$, we get isomorphisms
\[
\Ext^i_P(S,Q)\cong\Ext^i_Q(R,Q)
\]
for all $i$ (cf.\ \cite[2.1(1)]{JM}).  Thus, $\grade_P(S,Q)=\grade_Q(R,Q)$. This and the assumptions of the theorem now give
\begin{align*}
\grade_P(S,P) &\geq \grade_P(S,Q)\\
&= \grade_Q(R,Q)\\
&= \pd_Q R\\
&= \pd_P S
\end{align*}
which implies that $\grade_P(S,P)=\pd_P S$, and so $\varphi'$ is Cohen-Macaulay.

To show that $S$ is a Gorenstein ring it remains to prove that $\Ext_P^{\pd_P S}\left(S,P\right)\cong S$.
However, since $S$ is a perfect $P$-module, it is sufficient to show that the rank of the last nonzero free module in
the minimal free resolution $\mathbf{F}$ of $S$ over $P$ is one. But this follows from the fact that $\mathbf{F}\otimes_P Q$
is a minimal free resolution of $R$ over $Q$ whose last nonzero free module has rank one.

Next we justify that $S$ is non-regular.  Let $I=\ker\varphi'$, which by assumption is contained in $\fm_P^2$.  We
have $\fm_S/\fm_S^2= \fm_P/(\fm_P^2+I)=\fm_P/\fm_P^2$.  Since the Cohen-Macaulayness of $S$ and $P$, along with the Auslander-Buchsbaum formula, imply that $\dim S<\dim P$, we therefore have
\[
\mu_S(\fm_S)=\mu_P(\fm_P)\ge \dim P >\dim S
\]
and so $S$ is non-regular.

Finally we show that the map $\psi'$ has finite flat dimension. Let $M$ be any $S$-module. By the vanishing
of $\Tor_i^P(S,Q)$ for each $i>0$, we have isomorphisms $\Tor_i^S(R,M)\cong\Tor_i^P(Q,M)$ for all $i$.
Since $Q$ has finite flat dimension as a $P$-module, this homology eventually vanishes.  The finite flat dimension of $R$ over $S$ follows.
\end{proof}

\begin{remark} Theorem~\ref{thm:main} may be regarded as a sort of `ascent' analogue of \cite[Theorem 3.1(2)]{Vel}.  Both results use a similar factorization
involving Gorenstein homomorphisms, whereas the point of the theorem from \cite{Vel} is that the hypotheses placed on $\varphi'$ descend to $\varphi$.
\end{remark}

We list as corollaries several cases in which Theorem \ref{thm:main} applies to establish the existence of non-trivial totally acyclic complexes over $R$.  We shall begin by showing that our result recovers the class of local rings which have embedded deformations.  Our proof of this result, as well as that of Corollary \ref{cor:grade3} below, use standard constructions of ring homomorphisms, as, for example,  seen in \cite{Vel}.

\begin{proof}[Proof of Theorem \ref{thm:embdef}] Let $\varphi:Q\ra Q/(x_1,\dots,x_c)\cong R$ define an embedded
deformation for $R$, so that $x_1,\dots,x_c$ is a $Q$-regular sequence contained in $\fm_Q^2$. Choose a minimal
system of generators $z_1,\dots,z_e$ of $\fm_Q$ and, for $1\le i \le c$, write $x_i=\sum_{j=1}^e r_{ij}z_j$
where each $r_{ij}\in \fm_Q$.

Now let $\rho_{ij}$, $\xi_{j}$ for $1\le i\le c$ and $1\le j\le e$ be indeterminates over $\mathbb Z$.  If
$p=\ch Q/\fm_Q$, set \[ P=\mathbb Z[\rho_{ij},\xi_{j}]_{(p,\rho_{ij},\xi_{j})} \] and consider the local
ring homomorphism $\psi:P\ra Q$ given by $\rho_{ij}\mapsto r_{ij}$, and $\xi_{j}\mapsto z_j$ for all $i$ and
$j$.  Thus, we define $\chi_i=\sum_{j=1}^e\rho_{ij}\xi_{j}$ for each $1\leq i\leq c$ and let \[
S=P/\left(\chi_1,\ldots,\chi_c\right). \] It now suffices to show that $S$ satisfies the hypotheses of
Theorem \ref{thm:main}.

First of all, notice that by way of construction, $S$ is a lifting of $R$ to $P$ via $Q$.  To see this more
explicitly, notice that \[ P/(\chi_1,\ldots,\chi_c)\otimes_P Q\xra{\cong} P\otimes_P Q/(x_1,\ldots,x_c) \]
as $P$-modules via the mapping given by $\overline{a}\otimes b\mapsto 1\otimes \overline{\psi(a)b}$, whence
$S\otimes_P Q\cong R$.  Furthermore, as $x_1,\ldots,x_c$ is $Q$-regular and $\chi_1,\ldots,\chi_c$ is
$P$-regular, if $\mathbf{F}$ is a $P$-free resolution of $S$, then $\mathbf{F}\otimes_P Q$ yields a $Q$-free resolution of $R$.  This implies Tor-independence.

The required grade inequality is actually a grade equality in this case,
\[
\grade_P(S,P)= c =\grade_P(S,Q).
\]
The remaining hypotheses are obvious, and the result follows. \end{proof}

\begin{corollary}\label{cor:grade3} Suppose that $\varphi:Q\ra R$ is a Gorenstein homomorphism of grade 3.  Then there exist
non-trivial totally acyclic complexes over $R$. \end{corollary}

\begin{proof}  By the Buchsbaum-Eisenbud structure theorem (cf.\ \cite[Theorem 3.4.1(b)]{BH}), there exists a
(deleted) minimal free resolution of $R$ over $Q$ given by
\[
\mathbf{F}:\qquad 0\ra Q\xra{\beta}Q^d\xra{\alpha}Q^d\xra{\beta^*}Q\ra 0
\]
 with $d$ odd, where $\alpha=(a_{ij})$ is skew-symmetric
and $\beta=(b_j)$, such that $b_j$ is determined by the Pfaffian of the matrix obtained by the deletion of
the $j^{th}$ row and column of $\alpha$. Notice that $\beta$ is non-trivial since $d$ is assumed to be odd.

Through a similar process as that in the previous proof, we define a regular local ring
\[
P=\ZZ[T]_{(p;T)}
\]
where $T=\{t_{ij}\var 1\leq i<j\leq d\}$ is a set of indeterminates over $\ZZ$ and $p=\ch Q/\fm_Q$, and a
local homomorphism $\psi: P\ra Q$ which acts by $t_{ij}\mapsto a_{ij}$ for each $i$ and $j$.

Now consider the $d\times d$ matrix over $P$ given by $\tau=(t_{ij})$ where $t_{ii}=0$ and $t_{ji}=t_{ij}$
for all $i<j$.  Let $\sigma=(s_{j})$ be the $d\times 1$ matrix over $P$ such that $s_j$ is defined by the
Pfaffian of the matrix obtained by deleting the $j^{th}$ row and column of $\tau$.  If $S$ is the $P$-module
defined by the cokernel of ${\sigma}^*$, the Buchsbaum-Eisenbud structure theorem implies that
\[
\mathbf{G}:\qquad 0\ra P\xra{\sigma}P^d\xra{\tau}P^d\xra{\sigma^*}P\ra 0
\]
is a (deleted) minimal free resolution of $S$ over
$P$.  With these defined, notice that $\mathbf{G}\otimes_P Q\cong\mathbf{F}$, which implies that $R$ lifts to $P$ via $Q$.  We can now use Theorem \ref{thm:main} to obtain the desired result. \end{proof}

The next result addresses the case that $Q$ is a Cohen-Macaulay ring.  Its proof uses a corollary to Robert's ``new intersection theorem'' (cf.\ \cite{Rob}), which we state as a lemma.

\begin{lemma}\label{lem:intersection} Let $R$ be a local ring, and suppose that $M$ and $N$ are finitely generated
$R$-modules of finite projective dimension over $R$.  If $\Tor_i^R(M,N)=0$ for all $i>0$ and  $M\otimes_R N$
is Cohen-Macaulay, then both $M$ and $N$ are Cohen-Macaulay. \end{lemma}

\begin{proof} Recall that the vanishing of $\Tor_i^R(M,N)$ for $i>0$ implies that
\[
\pd_R(M\otimes_R N)=\pd_R M+\pd_R N.
\]
Furthermore, by using \cite[Corollary 9.4.6]{BH} along with the Auslander-Buchsbaum
formula, we obtain the following:
\begin{align*}
\dim_R N &\leq \pd_R M +\dim_R(M\otimes_R N)\\
&=\pd_R(M\otimes_R N) -\pd_R N +\depth_R(M\otimes_R N)\\
&=\depth_R R -\pd_R N\\
&=\depth_R N
\end{align*}
The same proof shows that $M$ is also Cohen-Macaulay.
\end{proof}

\begin{corollary} \label{cor:cm}
Let $\varphi:Q\ra R$ be a Gorenstein homomorphism of local rings whose kernel is contained
in $\fm_Q^2$. Suppose that there exists a Gorenstein local ring $P$ and Cohen-Macaulay homomorphism $\psi:P\ra Q$
whose kernel is contained in $\fm_P^2$. If $R$ lifts to $P$ via $Q$, then there exist non-trivial
totally acyclic complexes over $R$. \end{corollary}

\begin{proof} If $P$ is non-regular, then by virtue of the fact that $\pd_P Q$ and $\pd_Q R$ are both
finite, we have that $\pd_P R$ is finite as well.  Thus Lemma \ref{lem:finflatdim} shows that non-trivial totally
acyclic complexes exist over $R$.  The rest of the proof addresses the case that $P$ is regular.

Let $S$ be  a lifting of $R$ to $P$. First, we want to show that $S$ is a quotient ring of $P$, equivalently
a cyclic $P$-module. We have that the induced map $P/\fm_P\ra Q/\fm_Q$ is an isomorphism, and therefore
\begin{align*}
R/\fm_Q R &\cong R\otimes_Q Q/\fm_Q\\
&\cong (S\otimes_P Q)\otimes_Q Q/\fm_Q\\
&\cong S\otimes_P Q/\fm_Q\\
&\cong S\otimes_P P/\fm_P\\
&\cong S/\fm_P S
\end{align*}
which implies that $R/\fm_Q R\cong S/\fm_P S$ as vector spaces over $Q/\fm_Q$.  It follows that
\[
\mu_P(S)=\dim_{P/\fm_P}S/\fm_P S=\dim_{Q/\fm_Q}R/\fm_Q R=1.
\]

Now consider the natural projection $\varphi':P\ra S$, and the natural local ring homomorphism $\psi':S\ra S\otimes_P Q\cong R$
which acts by $s\mapsto s\otimes 1$.  Since $\psi$ Cohen-Macaulay implies $\fd_P Q<\infty$, it is sufficient to show that $S$ satisfies the hypotheses of Theorem \ref{thm:main}. To this end, we prove that $\grade_P(S,P)\geq\grade_P(S,Q)$ and the kernel of $\varphi'$ is contained in $\fm_P^2$.

In order to show the latter, let $x\in\ker\varphi'$.  Then $\psi(x)\in\ker\varphi\subseteq\fm_Q^2$. Since the induced map $\fm_P/\fm_P^2\ra\fm_Q/\fm_Q^2$ is injective, it follows that $x\in\fm_P^2$.  Now to show the former, note that $Q$ is Cohen-Macaulay since it is perfect module over a Cohen-Macaulay ring. This fact, in turn,
implies the same property for $R$.  Recall that Lemma \ref{lem:intersection} shows that $S$ must be
Cohen-Macaulay as well.  Since $\pd_P S<\infty$, it follows that $\varphi'$ is Cohen-Macaulay.  This fact and the lifting condition imply the following equalities
\[
\grade_Q(R,Q)=\pd_Q R=\pd_P S=\grade_P(S,P).
\]
Recalling that vanishing of $\Tor_i^R(S,Q)$ for $i>0$ implies that $\grade_P(S,Q)=\grade_Q(R,Q)$ (see the proof of Theorem \ref{thm:main}), the desired grade condition is satisfied.  We can now apply Theorem \ref{thm:main} to obtain the result.
\end{proof}

In the previous corollary, the assumption that $\ker\psi\subseteq\fm_P^2$ is essential in obtaining totally
reflexive $R$-modules which are \emph{non-trivial}.  This fact is illustrated through the following
example.

\begin{example} Let $k$ be a field and consider the local rings defined by:
\[
P=k\dbl x,y\dbr\quad Q=k\dbl x\dbr\quad R=k\dbl x\dbr/(x^2)
\]
Furthermore, let $\psi:P\ra k\dbl x,y\dbr/(y-x^2)\cong Q$ and
$\varphi:Q\ra R$ be the natural projection maps.  Notice that $\varphi\circ\psi:P\ra R$ can be alternately
factored to obtain the following commutative diagram of local rings
\[
\xymatrixrowsep{2pc}
\xymatrixcolsep{2pc}
\xymatrix{ & P\ar@{->}_{\varphi'}[dl]\ar@{->}^{\psi}[dr]\\
S\ar@{->}_{\psi'}[dr] & & Q\ar@{->}^{\varphi}[dl]\\
& R }
\]
where $S=k\dbl x,y\dbr/(y)\cong k\dbl x\dbr$ and $\psi'$ and $\varphi'$
are the natural projection maps.  Though $S$ is clearly a lifting of $R$ to $P$ via $Q$, $R$ does not fit
the criteria for Theorem \ref{thm:main} as $\ker\psi\nsubseteq\fm_P^2$.  The consequence lies in the fact that
$Q$ is regular, and thus all of its totally reflexive modules are in fact free.  The induced $R$-modules
will therefore be free as well. \end{example}

\begin{corollary} Let $\varphi:Q\ra R$ be a Gorenstein homomorphism of local rings whose kernel is contained
in $\fm_Q^2$. Suppose that $P$ is a Gorenstein local ring and $\psi:P\ra Q$ a local homomorphism of finite
flat dimension such that \begin{enumerate} \item the induced map $\fm_P/\fm_P^2\ra\fm_Q/\fm_Q^2$ is
injective and \item the induced map $P/\fm_P\ra Q/\fm_Q$ is bijective. \end{enumerate} If there exists a
Cohen-Macaulay lifting of $R$ to $P$ via $Q$, then there exist non-trivial totally acyclic complexes over
$R$. \end{corollary}

\begin{proof}
Let $S$ be such a lifting of $R$ to $P$ via $Q$.  Since $S$ is Cohen-Macaulay and the lifting of $R$ implies that $\pd_P S$ is finite, we have that $S$ is a perfect $P$-module.  Using this fact, it is easy to follow the same steps as in the proof of Corollary \ref{cor:cm} to verify that $S$ satisfies the hypotheses necessary for the application of Theorem \ref{thm:main}.
\end{proof}

\section{Examples}

In this section we turn our attention to examples of rings which admit totally acyclic complexes by virtue of Theorem \ref{thm:main}. In fact, we are able to construct such rings with associated Gorenstein homomorphisms of arbitrary grade, which furthermore do not have embedded deformations.

\begin{construction} Let $k$ be a field, and define $P=k[X_1,\ldots,X_n,Y_1,\ldots,Y_m]_\fm$ where $\fm=(X_1,\ldots,X_n,Y_1,\ldots,Y_m)$ is the homogeneous maximal ideal over the polynomial ring $k[X_1,\ldots,X_n,Y_1,\ldots,Y_m]$.  Furthermore, let $\fm_X=(X_1,\ldots,X_n)$ and $\fm_Y=(Y_1,\ldots,Y_m)$ denote the homogeneous maximal ideals over the polynomial rings $k[X_1,\ldots,X_n]$ and $k[Y_1,\ldots,Y_m]$, respectively.  Now let
\[
\begin{array}{ll}
f_i\in P_X:=k[X_1,\ldots,X_n]_{\fm_X} & 1\leq i\leq r\\[0.05in]
g_j\in P_Y:=k[Y_1,\ldots,Y_m]_{\fm_Y} & 1\leq j\leq s \end{array}
\]
each be contained in the square of the respective maximal ideals.  If $(f_1,\ldots,f_r)P$ is a Gorenstein ideal of $P$ and
$(g_1,\ldots,g_s)P$ is a perfect ideal of $P$, then
\[
R=P/(f_1,\ldots,f_r,g_1,\ldots,g_s)P
\]
admits non-trivial totally acyclic complexes. Moreover, if $(g_1,\ldots,g_s)P$ is chosen
to be a non-Gorenstein ideal $P$, then $R$ is a non-Gorenstein ring.
\end{construction}

To justify these claims, let $S=P/(f_1,\ldots,f_r)P$ and $Q=P/(g_1,\ldots,g_s)P$, and notice that $R\cong S\otimes_P Q$.  We need to show that $S$ and $Q$ are Tor-independent $P$-modules, that $\grade_P(S,P)\geq\grade_P(S,Q)$, and that the projection $Q\to R$ is a Gorenstein homomorphism.  These facts are illustrated below.  In order to make notation more concise, we let $(\mathbf{f})$ and $(\mathbf{g})$ denote the ideals $(f_1,\ldots,f_r)P_X$ and $(g_1,\ldots,g_s)P_Y$, respectively.  Furthermore, unless otherwise stated, all tensor products are assumed to be taken over $k$.

First we show that $\Tor_i^R(S,Q)$ vanishes for positive $i$.  Take $\mathbf{F}\to P_X/(\mathbf{f})\to 0$ and $\mathbf{G}\to P_Y/(\mathbf{g})\to 0$ to be a free resolutions over $P_X$ and $P_Y$, respectively.  Then $(\mathbf{F}\otimes P_Y)_{\tilde\fm}$ and $(P_X\otimes \mathbf{G})_{\tilde\fm}$ are free resolutions of $S$ and $Q$, respectively, over $P$, where we define $\tilde\fm=\fm_X\otimes P_Y+P_X\otimes \fm_Y$.  To see that $S$ and $Q$ are Tor-independent over $P$, notice that
\[
\Tor_i^P(S,Q)=\HH_i\left((\mathbf{F}\otimes P_Y)_{\tilde\fm}\otimes_P(P_X\otimes \mathbf{G})_{\tilde\fm}\right)\cong\HH_i(\mathbf{F}\otimes \mathbf{G})_{\tilde\fm}
\]
where the isomorphism is obtained from \cite[2.2]{J}.  Since this homology is isomorphic to $R$ for $i=0$ and vanishes otherwise, we have the lifting of $R$ to $P$ via $Q$.

Next we establish the grade inequality.  As $S=P/(\mathbf{f})P$ is Gorenstein and $Y_1,\ldots,Y_m$ is regular on $S$, we have that
\[
S/(Y_1,\ldots,Y_m)S\cong P_X/(\mathbf{f})
\]
is also Gorenstein; in particular, $P_X/(\mathbf{f})$ is perfect as a module over $P_X$.  This fact implies the last of the following equalities:
\[
\grade_P(S,P)=\pd_P S=\pd_{P_X}P_X/(\mathbf{f})=\grade_{P_X}(P_X/(\mathbf{f}),P_X)
\]

Furthermore, note the following isomorphisms of complexes
\begin{align*}\tag{\ddag}
\begin{split}
\Hom_P((\mathbf{F}\otimes P_Y)_{\tilde\fm},Q)&\cong \Hom_{P_X\otimes P_Y}(\mathbf{F}\otimes P_Y,P_X\otimes P_Y/(\mathbf{g}))_{\tilde\fm}\notag\\
&\cong(\Hom_{P_X}(\mathbf{F},P_X)\otimes\Hom_{P_Y}(P_Y,P_Y/(\mathbf{g})))_{\tilde\fm}\\
&\cong(\Hom_{P_X}(\mathbf{F},P_X)\otimes P_Y/(\mathbf{g}))_{\tilde\fm}\notag
\end{split}
\end{align*}
the second of which is obtained from \cite[Proof of Lemma 2.5(1)]{J}.  Now since
\begin{align*}
\HH\left(\Hom_{P_X}(\mathbf{F},P_X)\otimes P_Y/(\mathbf{g})\right)&=\Ext^{\pd_P S}_{P_X}(P_X/(\mathbf{f}),P_X)\otimes P_Y/(\mathbf{g})\\
&\cong P_X/(\mathbf{f})\otimes P_Y/(\mathbf{g})
\end{align*}
is nonzero upon localizing at $\tilde\fm$, we have $\grade_P(S,Q)=\grade_{P_X}(P_X/(\mathbf{f}),P_X)$, and therefore $\grade_P(S,P)=\grade_P(S,Q)$.  In particular, the inequality holds.

Finally, we verify that $Q\to R$ is Gorenstein.  To do this, recall that the vanishing of $\Tor_i^P(S,Q)$ for $i>0$ implies that $\Ext_Q^i(R,Q)=\Ext_P^i(S,Q)$ for all $i\in\NN$.  Thus,
\[
\grade_Q(R,Q)=\grade_P(S,Q)=\grade_P(S,P).
\]
Since, furthermore, $\pd_Q R=\pd_P S$ due to the lifting of $R$, we have established that $R$ is a perfect $Q$-module.  To verify that $Q\to R$ is Gorenstein, it is enough check that $\Ext_Q^{\pd_Q R}(R,Q)\cong R$.  However, this is equivalent to checking the same for $\Ext_P^{\pd_P S}(S,Q)$, which is obtained by taking homology of ($\ddag$).  To this end, we calculate:
\[
\Ext_P^{\pd_P S}(S,Q)\cong\left(P_X/(\mathbf{f})\otimes P_Y/(\mathbf{g})\right)_{\tilde\fm}\cong R
\]

Therefore Theorem \ref{thm:main} establishes the existence of non-trivial totally acyclic complexes over $R$.  In order to check the validity of the final statement, notice that if $(\mathbf{g})P$ is not a Gorenstein ideal of $P$, then $Q$ is Cohen-Macaulay but not Gorenstein.  Thus, the rank of the last nonzero free module in $\left(P_X\otimes \mathbf{G}\right)_{\tilde\fm}$ is greater than one.  Furthermore,
\[
S\otimes_P\left(P_X\otimes \mathbf{G}\right)_{\tilde\fm}\to R\to 0
\]
is a free resolution of $R$ over $S$, and its last nonzero free module must also have rank greater than one.  Since $S$ is assumed to be Gorenstein, we have shown that $R$ cannot be.
\medskip

This section concludes with a specific example of a ring $R$ which demonstrates the previous construction, and which, as we shall show in the next section, has no embedded deformation.

\begin{example}\label{ex:noembdef} Let $k$ be a field and $P=k[X_1,\ldots,X_5,Y_1,\ldots,Y_4]_\fm$, where $\fm=(X_1,\ldots,X_5,Y_1,\ldots,Y_4)$ is the homogeneous maximal ideal over the polynomial ring $k[X_1,\ldots,X_5,Y_1,\ldots,Y_4]$.  Now consider the local ring
$R=P/I$, where $I$ is defined by the following seventeen quadratics over $P$:\vspace{.1in}
\begin{center}
\mbox{$2X_1X_3+X_2X_3,\hspace{.05in} X_1X_4+X_2X_4,\hspace{.05in} X_3^2+2X_1X_5-X_2X_5$}\vspace{.075in}
\mbox{$X_4^2+X_1X_5-X_2X_5,\hspace{.05in} X_1^2,\hspace{.05in} X_2^2,\hspace{.05in} X_3X_4,\hspace{.05in}
X_3X_5,\hspace{.05in} X_4X_5,\hspace{.05in} X_5^2$}\vspace{.075in}\\
\mbox{$Y_1^2,\hspace{.05in}Y_1Y_2-Y_3^2,\hspace{.05in}Y_1Y_3-Y_2Y_4,\hspace{.05in}Y_1Y_4,\hspace{.05in}Y_2^2+Y_3Y_4,\hspace{.05in}Y_2Y_3,\hspace{.05in}Y_4^2$}
\end{center}\vspace{.1in}
We first notice that $R\cong S\otimes_PQ$, where $S=P/J$, $Q=P/K$, and $J$ and $K$ are the ideals generated
by\vspace{.1in}
\begin{center}
\mbox{$2X_1X_3+X_2X_3,\hspace{.05in} X_1X_4+X_2X_4,\hspace{.05in}
X_3^2+2X_1X_5-X_2X_5$}\vspace{.075in}
\mbox{$X_4^2+X_1X_5-X_2X_5,\hspace{.05in} X_1^2,\hspace{.05in}
X_2^2,\hspace{.05in} X_3X_4,\hspace{.05in} X_3X_5,\hspace{.05in} X_4X_5,\hspace{.05in} X_5^2$}
\end{center}\vspace{.05in}
and\vspace{.05in}
\[Y_1^2,\hspace{.05in}Y_1Y_2-Y_3^2,\hspace{.05in}Y_1Y_3-Y_2Y_4,\hspace{.05in}Y_1Y_4,\hspace{.05in}Y_2^2+Y_3Y_4,\hspace{.05in}Y_2Y_3,\hspace{.05in}Y_4^2\]
respectively, over $P$.  This noted, it is clear from the discussion in the previous
construction that $R$ fits the criteria for Theorem \ref{thm:main}, and so we are guaranteed that it admits
non-trivial totally acyclic complexes. However, as we will demonstrate in the subsequent section, $R$ does not have an embedded deformation.
\end{example}

\section{On The Homotopy Lie Algebra}

In this section, we investigate a method which uses the homotopy Lie algebra of a local ring to determine if it
has an embedded deformation.  Before giving results, we first review some related facts.

Recall that any surjective local homomorphism $Q\ra R$ induces a map $\pi^*(R)\ra\pi^*(Q)$ on
the respective graded homotopy Lie algebras.  If, furthermore, $Q$ is an embedded deformation of $R$, then
the natural map $\pi^*(R)\ra\pi^*(Q)$ is surjective and its kernel is comprised of the central elements of
$\pi^2(R)$; for details, see \cite[(6.1)]{A}.  For a local ring $(R,\fm,k)$ the universal enveloping algebra of $\pi^*(R)$ is precisely the graded $k$-algebra $\Ext_R^*(k,k)$.  If $R$ is moreover a Koszul algebra, then $\Ext_R^*(k,k)$ is generated as a $k$-algebra by $\Ext_R^1(k,k)$; for details, see \cite[Theorem 1.2]{L}.

The following result, which characterizes the algebra generated by $\Ext_R^1(k,k)$ for a quadratic ring $R$, is a special case of a result of L\"{o}fwall in \cite{L}.  Its proof has been omitted, but details can be found in \cite{L}.

\begin{theorem}\cite[Corollary 1.3]{L}\label{thm:lofwall}
Let $k$ be a field and define the $k$-algebra $R=k[x_1,\ldots,x_n]/(f_1,\ldots,f_r)$, where
\[
f_i=\sum_{j\leq\ell} a_{ij\ell}x_jx_\ell
\]
with each $a_{ij\ell}\in k$, are homogenous for $1\leq i\leq r$.  Then the algebra generated by the degree one elements in $\Ext^*_R(k,k)$ is given by
\[
\left[\Ext_R^1(k,k)\right]=k\langle T_1,\ldots,T_n\rangle/(\varphi_1,\ldots,\varphi_s)
\]
where, for $1\leq i\leq s$, we define
\[
\varphi_i=\sum_{j\leq \ell}c_{ij\ell}\left[T_j,T_\ell\right]
\]
such that $c_{ij\ell}\in k$ and $\left[T_j,T_\ell\right]=T_jT_\ell+T_\ell T_j$.  Furthermore, the $(c_{ij\ell})_{j\ell}$ form a basis for the solution set to the system of linear equations given by
\[
\sum_{j\leq \ell}a_{ij\ell}x_{j\ell}=0
\]
for $1\leq i\leq r$.  That is, $(c_{ij\ell})_{j\ell}$ forms a basis for the nullspace of the matrix given by:
\[
\left[\begin{array}{cccccc}
a_{111} & \cdots & a_{1nn}\\
\vdots & \ddots & \vdots\\
a_{r11} & \cdots & a_{rnn}
\end{array}\right]
\]
\end{theorem}

\medskip

As a consequence of L\"{o}fwall's result, we are able to consider degree two elements of the homotopy Lie algebra of a Koszul algebra $R$ as quadratic forms in $[\Ext_R^1(k,k)]$.  The following result illustrates this process in an extension of Theorem \ref{thm:lofwall}.

\begin{lemma}\label{lem:central}
Let $k$ be a field and consider the $k$-algebras given by
\begin{align*}
Q&=k[x_1,\ldots,x_n]/(f_1,\ldots,f_r)\\[.025in]
S&=k[y_1,\ldots,y_m]/(g_1,\ldots,g_s)
\end{align*}
where, for all $1\leq i\leq r$ and $1\leq j\leq s$, the $f_i,g_j$ are homogeneous forms such that $Q$ and $S$ are finite dimensional and Koszul.  If $R=Q\otimes_k S$ then $R$ is local and
\[
\pi^*(R)\cong\pi^*(Q)\times\pi^*(S).
\]
In particular, $\pi^*(R)$ has nonzero central elements of degree two if and only if either $\pi^*(Q)$ or $\pi^*(S)$ does.
\end{lemma}

\begin{proof}
Since the polynomials which define $Q$ and $S$ are homogeneous quadratics, we can express them
as
\[
f_i=\sum_{j\leq \ell}a_{ij\ell}x_j x_\ell
\]
for $1\leq i\leq r$, and
\[
g_i=\sum_{j\leq \ell}b_{ij\ell}y_j y_\ell
\]
for $1\leq i\leq s$.  Further, since
\[
R=Q\otimes_k S\cong k[x_1,\ldots,x_n,y_1,\ldots,y_m]/(f_1,\ldots,f_r,g_1,\ldots,g_s)
\]
we see that Theorem \ref{thm:lofwall} implies
\[
\left[\Ext_R^1(k,k)\right]=k\langle T_1,\ldots,T_n,U_1,\ldots,U_m\rangle/(\varphi_1,\ldots,\varphi_\alpha)
\]
where $\alpha=(n+m)(n+m+1)/2-(r+s)$.  For simplicity in notation, we let $\beta=n(n+1)/2-r$ and
$\gamma=m(m+1)/2-s$, so that $\alpha=\beta+\gamma+nm$.  With this established, the first $\beta+\gamma$ of the $\varphi_i$ are given by
\[
\varphi_i=\left\{\begin{array}{ll} \displaystyle{\sum_{j\leq \ell}c_{ij\ell}[T_j,T_\ell]} & 1\leq
i\leq\beta\vspace{.05in}\\
\displaystyle{\sum_{j\leq \ell}c_{ij\ell}[U_j,U_\ell]} & \beta+1\leq
i\leq\beta+\gamma\vspace{.05in} \end{array}\right.
\]
where $c_{ij\ell}\in k$ is defined in such a way that $(c_{ij\ell})_{j\ell}$ forms basis for the kernel of the $(r+s)\times(\beta+\gamma)$ matrix given by:
\[
\left[\begin{array}{cccccc} a_{111} & \cdots & a_{1nn} &&& \\
\vdots & \ddots & \vdots && \mbox{\Huge{0}} &\\
a_{r11} & \cdots & a_{rnn} &&& \\
&&& b_{111} & \cdots & b_{1mm} \\
& \mbox{\Huge{0}} && \vdots & \ddots
& \vdots \\
&&& b_{s11} & \cdots & b_{smm} \end{array} \right]
\]
Furthermore, let
\[
\sigma:\left\{(j,\ell)\in\ZZ^2\var 1\leq j\leq n,\hspace{.03in}1\leq \ell\leq m\right\}\ra\left\{i\in\ZZ\var 1\leq i\leq nm\right\}
\]
be a bijection; then the last $nm$ of the $\varphi_i$ are given by
\[
\varphi_{\beta+\gamma+\sigma\left((j,\ell)\right)}=\left[T_j,U_\ell\right]
\]
for $1\leq j\leq n$ and $1\leq \ell\leq m$.  Letting $t_i$ (resp.\ $u_i)$ denote the image in $\Ext_R^1(k,k)$ of $T_i$ (resp.\ $U_i)$, it
follows that $[t_j,u_\ell]=[u_\ell,t_j]=0$ for every $1\leq j\leq n$ and $1\leq \ell\leq m$.  Now it follows that
\[
[\Ext_R^1(k,k)]\cong k\langle T_1,\ldots,T_n\rangle/(\varphi_1,\ldots,\varphi_\beta)\otimes_k k\langle U_1,\ldots,U_m\rangle/(\varphi_{\beta+1},\ldots,\varphi_{\beta+\gamma}).
\]
Furthermore as $R$ is assumed to be Koszul, $\pi^*(R)$ can be viewed as a linear subspace of the above expression via the natural inclusion $\pi^*(R)\hookrightarrow\Ext_R^*(k,k)$.  This implies the result.
\end{proof}

\begin{remark}
The statement of Lemma \ref{lem:central} holds even when $Q$, $S$, and $R$ are local, but not necessarily Koszul.  To see this, notice that we have an induced isomorphism $\Tor_R^*(k,k)\cong\Tor_Q^*(k,k)\otimes_k\Tor_S^*(k,k)$ of $k$-algebras which extends to an isomorphism of Hopf algebras with divided powers.  Moreover, one can show that this isomorphism is equivalent to an isomorphism of homotopy Lie algebras by considering the equivalence of the respective categories. (For details, see \cite{An}, \cite{MM}, \cite{S}.)  Despite this more general fact, we have chosen to include the machinery of Theorem \ref{thm:lofwall} and Lemma \ref{lem:central} so that we may justify the following result without needing the rigor which is required of Hopf algebras.
\end{remark}

With these results established, we are now ready to prove the assertion at the end of the previous section.

\begin{fact}
The local ring $R$ defined in Example \ref{ex:noembdef} does not have an embedded deformation.
\end{fact}

\begin{proof}
It suffices to show that $\pi^*(R)$ has no non-trivial central elements of degree 2.  By Lemma \ref{lem:central}, this condition is equivalent to neither $\pi^*(S)$ nor $\pi^*(Q)$ containing such elements.  In \cite[Example 2.1 \& Section 3]{AGP2}, Avramov, Gasharov, and Peeva prove this condition for $\pi^*(Q)$, so we only  need to show the result for $\pi^*(S)$.  We shall adopt the same approach as the authors of \cite{AGP2}.

As a result of \cite[Corollary 1.3]{L}, we know that the algebra generated by the universal enveloping algebra of $\pi^*(S)$ can be expressed as
\[
\left[\Ext_S^1(k,k)\right]\cong k\langle T_1,T_2,T_3,T_4,T_5\rangle/I
\]
where $I$ is generated by
\begin{center}
\mbox{$T_1T_2+T_2T_1,\quad(T_1T_3+T_3T_1)-2(T_2T_3+T_3T_2),\quad (T_1T_4+T_4T_1)-(T_2T_4+T_4T_2)$}\\[0.05in]
\mbox{$T_3^2+T_4^2+(T_2T_5+T_5T_2),\quad T_3^2+(T_1T_5+T_5T_1)+(T_2T_5+T_5T_2)$.}
\end{center}

So it follows that $\pi^*(S)$ is a graded Lie algebra on the variables $t_1,t_2,t_3,t_4,t_5$, each of degree one, which satisfies the relations
\begin{center}
\mbox{$[t_1,t_2]=0,\quad[t_1,t_3]=2[t_2,t_3],\quad [t_1,t_4]=[t_2,t_4]$}\\[0.05in]
\mbox{$t_3^{(2)}+t_4^{(2)}=-[t_2,t_5],\quad 2t_3^{(2)}+t_4^{(2)}=[t_1,t_5]$.}
\end{center}

It is straightforward to see that the following forms a basis of $\pi^2(S)$:
\[
\begin{array}{lllll}
u_1=t_1^{(2)} & u_2=t_2^{(2)} & u_3=t_3^{(2)} & u_4=t_4^{(2)} & u_5=t_5^{(2)}\\[0.05in]
u_6=[t_1,t_3] & u_7=[t_1,t_4] & u_8=[t_3,t_4] & u_9=[t_3,t_5] & u_{10}=[t_4,t_5]
\end{array}
\]

Furthermore, we assert that a basis of $\pi^3(S)$ is given by:
\[
v_i=\left\{\begin{array}{ll}
[u_{i+2},t_1] & 1\leq i\leq 8 \\[3pt]
[u_{i-2},t_3] & 9\leq i\leq 12 \\[3pt]
[u_{i-5},t_4] & 13\leq i\leq 14 \\[3pt]
[u_{i-6},t_5] & 15\leq i\leq 16
\end{array}\right.
\]
For the reader's convenience, and in order to justify these claims, we include a multiplication table for $\pi^3(S)$.

\medskip

\begin{center}
{
\renewcommand{\arraystretch}{1.5}
\begin{tabularx}{0.95\textwidth}{|c|C|C|C|C|C|}
\hline
$[u_i,t_j]$ & $t_1$ & $t_2$ & $t_3$ & $t_4$ & $t_5$\\
\hline
$u_1$ & 0 & 0 & $-2v_4$ & $-2v_5$ & $-4v_1+2v_2$\\
\hline
$u_2$ & 0 & 0 & $\frac{1}{2}v_4$ & $2v_5$ & $v_1+v_2$\\
\hline
$u_3$ & $v_1$ & $\frac{1}{2}v_1$ & 0 & $-2v_{10}$ & $-2v_{11}$\\
\hline
$u_4$ & $v_2$ & $\frac{1}{2}v_2$ & $-2v_{13}$ & 0 & $-\frac{1}{2}v_3+4v_{11}$\\
\hline
$u_5$ & $v_3$ & $-v_3+4v_{11}$ & $-2v_{15}$ & $-2v_{16}$ & 0\\
\hline
$u_6$ & $v_4$ & $-\frac{1}{2}v_4$ & $-\frac{1}{2}v_1$ & $-v_9-v_6$ & $-v_7+2v_{13}$\\
\hline
$u_7$ & $v_5$ & $-v_5$ & $v_9$ & $-\frac{1}{2}v_2$ & $-v_8+4v_{10}$\\
\hline
$u_8$ & $v_6$ & $\frac{1}{2}v_6-\frac{1}{2}v_9$ & $v_{10}$ & $v_{13}$ & $-v_{12}-v_{14}$\\
\hline
$u_9$ & $v_7$ & $\frac{1}{2}v_7-3v_{13}$ & $v_{11}$ & $v_{14}$ & $v_{15}$\\
\hline
$u_{10}$ & $v_8$ & $v_8-6v_{10}$ & $v_{12}$ & $\frac{1}{4}v_3-2v_{11}$ & $v_{16}$\\
\hline
\end{tabularx}
}
\end{center}

\medskip

It is clear from this table that the elements $v_1,\ldots,v_{16}$ span $\pi^3(S)$.  In order to justify their linear independence, we note that $\rank_k\pi^3(S)=\varepsilon_3(S)$, the third deviation of $S$, cf.\ \cite[Theorem 10.2.1(2)]{A2}.  This quantity can be calculated in terms of the Betti numbers of $k$ over $S$ as follows
\begin{align*}\tag{$\sharp$}
\begin{split}
\varepsilon_1&=b_1\notag\\
\varepsilon_2&=b_2-\binom{\varepsilon_1}{2}\\
\varepsilon_3&=b_3-\varepsilon_2\varepsilon_1-\binom{\varepsilon_1}{3}\notag
\end{split}
\end{align*}
(cf.\ \cite[Section 7]{A2}).  Calculating a minimal $S$-free resolution of $k$ yields that
\[
P^S_k(t)=1+5t+20t^2+76t^3+\cdots
\]
which we use to evaluate the expressions in ($\sharp$), and obtain $\varepsilon_3=16$.  Thus, $v_1,\ldots,v_{16}$ is in fact a basis of $\pi^3(S)$.

Now suppose $u=\sum_{i=1}^{10}\alpha_i u_i$ is central in $\pi^2(S)$.  Then $0=[u,t_1]=\sum_{i=3}^{10}\alpha_i v_{i-3}$ implies that $u=\alpha_1 u_1+\alpha_2 u_2$.  Furthermore, using the above table yields
\begin{align*}
0 &= [u,t_5]\\
&= \alpha_1[u_1,t_5]+\alpha_2[u_2,t_5]\\
&= (-4\alpha_1+\alpha_2)v_1+(2\alpha_1+\alpha_2)v_2
\end{align*}
which implies that $u=0$.  We have therefore proven that $\pi^2(S)$ does not contain nonzero central elements, and thus $R$ does not have an embedded deformation.
\end{proof}

Recalling that (Gorenstein) local rings of codimension at most 3 (resp.\ 4) have embedded deformations, we
have that the ring defined in Example \ref{ex:noembdef}, in the way of codimension, the smallest such
possible which satisfies the hypotheses of our main result, yet does not have an embedded deformation.

\section*{Acknowledgements}

I would like to thank Luchezar Avramov for providing many helpful remarks.  I am moreover indebted to my thesis advisor, David Jorgensen, not only for not reading many versions of this manuscript and offering invaluable suggestions, but also for several interesting conversations.  Finally, I am grateful to the referee for recommendations which improved the overall clarity of this paper.

\end{document}